\documentclass[12pt]{amsart}

\usepackage{amsfonts,amssymb,stmaryrd,amscd,amsmath,latexsym,amsbsy}

\newtheorem{theorem}{Theorem}[section]
\newtheorem{lemma}[theorem]{Lemma}
\newtheorem{proposition}[theorem]{Proposition}
\newtheorem{corollary}[theorem]{Corollary}
\theoremstyle{definition}
\newtheorem{definition}[theorem]{Definition}

\newtheorem{example}[theorem]{Example}

\newtheorem{remark}[theorem]{Remark}

\newcommand{\ext}{\text{Ext}}

\newcommand{\End}{\text{End}}
\newcommand{\Hom}{\text{Hom}}

\newcommand{\Rep}{\text{Rep}}

\newcommand{\C}{\mathcal{C}}

\newcommand{\ben}{\begin{enumerate}}
\newcommand{\een}{\end{enumerate}}

\theoremstyle{plain}

\newtheorem*{sol}{Solution}

\theoremstyle{definition}

\theoremstyle{remark}

\newcommand{\solu}[1]{\begin{sol}{\bf (\ref{#1})}}

\def\C{\mathcal{C}}
\def\D{\mathcal{D}}

\def\End{\mathrm{End}}
\def\Hom{\mathrm{Hom}}

\def\ext{\mathop{\mathrm{Ext}}\nolimits}

\def\Vec{\mathrm{Vec}}

\def\k{\mathbf{k}}

\def\Rep{\mathop{\mathrm{Rep}}\nolimits}

\def\FPdim{\mathop{\mathrm{FPdim}}\nolimits}

\begin{document}

\title{Frobenius-Perron dimensions of integral $\Bbb Z_+$-rings and applications}

\begin{abstract} We introduce the notion of the Frobenius-Perron dimension of an 
integral $\Bbb Z_+$-ring and give some applications of this notion to classification of 
finite dimensional quasi-Hopf algebras with a unique nontrivial simple module, 
and of quasi-Hopf and Hopf algebras of prime dimension $p$.
\end{abstract}

\author{Pavel Etingof}
\address{Department of Mathematics, Massachusetts Institute of Technology,
Cambridge, MA 02139, USA}
\email{etingof@math.mit.edu}

\maketitle

\vskip .05in
\centerline{\bf To Nicol\'as Andruskiewitsch on his 60th birthday with admiration} 
\vskip .05in

\section{Introduction} 

The goal of this paper is to define and study the notion of the Frobenius-Perron dimension 
of an integral $\Bbb Z_+$-ring $A$ which is not necessarily a fusion ring 
(i.e., does not necessarily have a *-structure). Namely, if $N_i$ are the matrices of 
multiplication by the basis vectors of $A$ and $d_i$ are their largest positive eigenvalues, 
then we let $\bold p$ be the left eigenvector of $N_i$ with eigenvalues $d_i$, normalized to have positive 
integer entries with greatest common divisor $1$. Then we set ${\rm FPdim}(A):=\sum_i p_id_i$. 
For fusion rings $p_i=d_i$ so ${\rm FPdim}(A)=\sum_i d_i^2$, which coincides with the standard definition in \cite{EGNO}. Also, if $A$ is the Grothendieck ring of a finite tensor category $\C$ then ${\rm FPdim}(A)={\rm FPdim}(\C)/D$, where $D$ is the greatest common divisor of the dimensions of the indecomposable projectives of $\C$. 

We prove a number of properties of ${\rm FPdim}(A)$. In particular, we show that if $X\in A$ is a $\Bbb Z_+$-generator of $A$ and $r$ the rank of $A$ then $(2d)^{r-1}\ge {\rm FPdim}(A)$, where $d$ is the Frobenius-Perron dimension of a $\Bbb Z_+$-generator of $A$. 

We then give two applications of this notion. The first application is to classification of quasi-Hopf and Hopf algebras with two simple modules. We give a number of examples of such quasi-Hopf and Hopf algebras and prove a number of restrictions on their structure. The second application is to quasi-Hopf and Hopf algebras $H$ of prime dimension $p$. It is conjectured that any such quasi-Hopf algebra is commutative and cocommutative, but this is not known even in characteristic zero, and we prove some restrictions on the structure of $H$ if it is not commutative and cocommutative; e.g., we show that $H$ has to have at least four simple modules. In particular, it follows that 
a quasi-Hopf algebra of prime dimension $\le 31$ over any field is necessarily commutative and cocommutative. In the case of Hopf algebras, it is known that if $H$ is not commutative and cocommutative then $q>2$ and $\frac{p}{q}>4$ (\cite{NW}). We improve this bound to $\frac{p}{q+2}>\frac{14}{3}$ (Corollary \ref{9/2}). Also, we show that $\frac{p}{q+2}>{\rm min}(9,\phi(p))$, where $\phi$ is a certain function such that $\phi(p)\sim 2\sqrt{2} \left(\frac{\log p}{\log\log p}\right)^{1/2}$ as $p\to \infty$ (Theorem \ref{inner2}). This means that if $p$ is large enough, we must have $\frac{p}{q+2}>9$. Moreover, we show that if $H$ has a nontrivial simple module $X$ with $X\cong X^{**}$ then 
$\frac{p}{q+2}>\phi(p)$. 

The organization of the paper is as follows. In Section 2 we define the Frobenius-Perron dimension 
of an integral $\Bbb Z_+$-ring $A$ and prove basic properties of this notion. In Section 3 we prove the lower bound on the dimension of a $\Bbb Z_+$-generator of $A$. In Section 4 we prove some auxiliary lemmas. In Section 5 we  give applications to quasi-Hopf algebras with two simple modules. In Section 6 we prove some bounds for dimensions of 
tensor categories with no nontrivial invertible objects. 
In Section 7 we give applications to classification of quasi-Hopf algebras of prime dimension. 
Finally, in Section 8 we give applications to classification of Hopf algebras of prime dimension, improving the bound on the characteristic from \cite{NW}.  

\vskip .05in
{\bf Acknowledgements.} The author is grateful to C. Negron for useful discussions and to V. Ostrik, S.-H. Ng and X. Wang 
for corrections and comments on the draft of this paper. The work of the author was partially supported by the NSF
grant DMS-1502244.

\section{The Frobenius-Perron dimension of an integral $\Bbb Z_+$-ring} 

\subsection{Dimensions in integral $\Bbb Z_+$-rings.}
Let $A$ be a transitive unital $\Bbb Z_+$-ring of finite rank (\cite{EGNO}, Definitions 3.1.1, 3.3.1). This means that $A$ is a free finitely generated abelian group with basis $b_i, i\in I$ 
and a unital associative ring structure such that $b_ib_j=\sum_k N_{ij}^kb_k$ with $N_{ij}^k\in \Bbb Z_+$ and $b_0=1$ for some element $0\in I$, and for each $j,k$ there exists $i$ with $N_{ij}^k>0$. In this case we have a homomorphism ${\rm FPdim}: A\to \Bbb R$, determined uniquely by the condition that $d_i:={\rm FPdim}(b_i)>0$, called the {\it Frobenius-Perron dimension} (\cite{EGNO}, Subsection 3.3). 

Let $N_i$ be the matrix with entries $N_{ij}^k$. Let $\bold d=(d_i)$ be the column vector with entries $d_i$. Then we have $N_i\bold d=d_i\bold d$, i.e., $\bold d$ is a common (right) eigenvector of the matrices $N_i$. 
The Frobenius-Perron theorem implies that we also have a left eigenvector $\bold p=(p_i)$ with $p_i>0$ (defined uniquely up to a positive scalar) 
such that $\bold p N_i=d_i\bold p$. 

\begin{lemma}\label{easybound} For each $i,j,k$ one has $p_k\ge \frac{N_{ji}^k}{d_j}p_i$.
In particular, one has $p_k\ge p_0/d_k$. 
\end{lemma} 

\begin{proof} One has $d_jp_k=\sum_i N_{ji}^kp_i$, which implies the statement. 
\end{proof}

If $A$ is a fusion ring then $N_{i^*}=N_i^T$, so $\bold p$ is proportional to $\bold d$ and it is natural to normalize 
$\bold p$ by setting $\bold p=\bold d$; in this case the number ${\rm FPdim}(A):=\bold p\cdot \bold d=|\bold d|^2=\sum_i d_i^2$ is called the Frobenius-Perron dimension of $A$ (\cite{EGNO}, Section 3.3).
However, in general $\bold p$ does not come with a natural normalization, which is why in \cite{EGNO} the Frobenius-Perron dimension 
of a general $\Bbb Z_+$-ring of finite rank is not defined. 

Note that when $A$ is the Grothendieck ring of a finite tensor category $\C$ then one may take $p_i$ to be the Frobenius-Perron dimensions of the indecomposable projectives, and $\bold p\cdot\bold d=\sum_i p_id_i$ is then called the Frobenius-Perron dimension of $\C$. However, 
this normalization depends on the choice of $\C$: e.g., the categories $\C_1$ of representations of $\Bbb Z/2$ 
and $\C_2$ of representations of Sweedler's 4-dimensional Hopf algebra categorify the same $\Bbb Z_+$-ring but give different normalizations of $\bold p$. 

Now recall (\cite{EGNO}) that $A$ is said to be {\it integral} if ${\rm FPdim}$ lands in $\Bbb Z$. 
In this case, the vector $\bold p$ can be uniquely normalized in such a way that $p_i$ are positive integers 
whose greatest common divisor is $1$. Let us call this normalization the {\it canonical normalization}.

\begin{definition} The Frobenius-Perron dimension of $A$ is 
${\rm FPdim}(A):=\bold p\cdot \bold d=\sum_i p_id_i$, 
where $p_i$ are normalized canonically. 
\end{definition} 

In particular, we see that ${\rm FPdim}(A)\ge \sum_i d_i$.

Note that for fusion rings the canonical normalization just gives $\bold p=\bold d$ as above, so our definition of ${\rm FPdim}(A)$ agrees with the one in \cite{EGNO}. However, in general $p_0\ne 1$, as shown by the following example. 

\begin{example}\label{X2=4} Consider the ring $A$ with basis $1,X$ with $X^2=4$. 
Thus, $A$ is the Grothendieck ring of the representation category $\C=\Rep(U)$
of the restricted enveloping algebra $U$ over a field $\k$ 
of characteristic $2$ generated by primitive elements $x,y,z$ with $z$ central, 
$[x,y]=z$ and $x^2=y^2=0, z^2=z$. Then $d_0=1, d_1=2$, $p_0=2, p_1=1$. 
This shows that the normalization with $p_0=1$ is not very natural: this would give $p_1=1/2$, 
which is not an algebraic integer. 
\end{example} 

\subsection{Grothendieck rings of integral finite tensor categories}
Now assume that $A$ is the Grothendieck ring of a finite tensor category $\C$. Then $\C$ is integral, i.e., of the form $\Rep H$, where $H$ is a finite dimensional quasi-Hopf algebra (\cite{EGNO}). Thus the Frobenius-Perron dimensions of objects of $\C$ are their usual vector space dimensions, and ${\rm FPdim}(\C)=\dim H$. In this case $p_i$ divide the dimensions $\dim(P_i)$ of the indecomposable projectives $P_i$ but aren't necessarily equal to them. E.g., if $H$ is Sweedler's 4-dimensional Hopf algebra, then $p_0=p_1=1$ but the indecomposable projectives $P_0,P_1$ are 2-dimensional. It is easy to see that in general we have the following proposition.

\begin{proposition}\label{gcd} Let $\C$ be an integral finite tensor category with Grothendieck ring $A$, and $D$ be the greatest common divisor of dimensions of the projective objects of $\C$. Then ${\rm FPdim}(\C)=D\cdot {\rm FPdim}(A)$.  
\end{proposition}

We also have 

\begin{proposition}\label{p1} Let $X_i$ be the simple objects of $\C$. Then one has $p_iD\ge d_i$, and the equality holds for a given $i$ if and only if the object $X_i$ is projective. Otherwise, one has $p_iD\ge 2d_i$. 
\end{proposition} 

\begin{proof} The first statement follows from the fact that the projective cover $P_i$ of $X_i$ contains 
$X_i$ in its composition series. For the second statement, it is enough to note that by the Frobenius property of $\C$ we have $\dim P_i\ge 2\dim X_i$ if $X_i$ is not projective. 
\end{proof} 

\begin{example}\label{unip} A tensor category $\C$ is called {\it unipotent} if its only simple object is the unit object, i.e. 
$A=\Bbb Z$ or, equivalently, ${\rm FPdim}(A)=1$ (\cite{EG1}). In other words, $\C={\rm Rep}H$, where $H$ is a local finite dimensional quasi-Hopf algebra. By \cite{EGNO}, Corollary 4.4.2, for $\dim H>1$ this can happen only in characteristic $p>0$. Moreover, in this case the associated graded  ${\rm gr}(H)$ of $H$ with respect to the radical filtration is cocommutative \footnote{This is a straightforward generalization of \cite{W}, Proposition 2.2(7) to the quasi-Hopf case, and I thank 
S. Gelaki for pointing this out.}, i.e. the group algebra of a finite unipotent group scheme over $k$, so 
$\dim H={\rm FPdim}(\C)$ is a power of $p$ (and clearly any power can arise this way).  
\end{example}

\section{Lower bound on the dimension of a $\Bbb Z_+$-generator}

Let $A$ be an integral $\Bbb Z_+$-ring of rank $r>1$ 
and $b=\sum_i m_ib_i\in A$ with $m_i\ge 0$. Let us say that $b$ is a $\Bbb Z_+$-{\it generator} of $A$
if for some $n$ all coefficients of $1+b+...+b^n$ are positive. Let $d={\rm FPdim}(b)$ and $N={\rm FPdim}(A)$. 

\begin{theorem}\label{bound} (i) Let $\chi$ be the characteristic polynomial of $b$ on $A$, and $Q(z)=(z-d)^{-1}\chi(z)$. Then $Q(d)$ is a positive integer divisible by $N$. In particular, $Q(d)\ge N$. 

(ii) Let $s$ be the number of roots of $Q$ with positive imaginary part (counted with multiplicities). Then 
$$
d\ge N^{\frac{1}{r-1}}\left(1-\frac{s+1}{r}\right)^{1-\frac{s}{r-1}},
$$ 
 
(iii)
We have 
$$
d\ge N^{\frac{1}{r-1}}\left( \frac{\lceil \frac{r-1}{2}\rceil}{r}\right)^{\frac{\lceil \frac{r-1}{2}\rceil}{r-1}}. 
$$
In particular, $d\ge \frac{1}{2}N^{\frac{1}{r-1}}$. 
\end{theorem}

\begin{proof} 
(i) First note that by the Frobenius-Perron theorem, $d$ is a simple eigenvalue of $b$, so $Q(d)$ is a positive integer. It remains to show that $Q(d)$ is divisible by $N$. Let $\bold v$ be an integral row vector such that $\bold p\cdot \bold v=1$ (it exists since $p_i$ are relatively prime). We have $Q(b)\bold v=m\bold d$, where $m=(Q(b)\bold v)_0$ is an integer. Therefore 
$$
Q(d)=Q(d)\bold p\cdot \bold v=\bold p Q(b)\bold v=m\bold p\cdot \bold d=mN.
$$ 
This implies that $N$ divides $Q(d)$. 

(ii) Let $\lambda_i$, $i=1,...,r-2s-1$ be the real roots of $Q$, and $\mu_1,...\mu_s$ the roots with positive imaginary part. 
We have 
$$
Q(d)=\prod_{i=1}^{r-2s-1} (d-\lambda_i)\prod_{j=1}^s |d-\mu_j|^2\ge N.
$$
Let $\alpha_i=\lambda_i/d$ and $\beta_j=\mu_j/d$. Then 
$$
\prod_{i=1}^{r-2s-1} (1-\alpha_i)\prod_{j=1}^s |1-\beta_j|^2\ge \frac{N}{d^{r-1}}.
$$
This can be written as 
$$
\prod_{i=1}^{r-2s-1} (1-\alpha_i)\prod_{j=1}^s (1-2{\rm Re}\beta_j+|\beta_j|^2)\ge \frac{N}{d^{r-1}}.
$$
By the Frobenius-Perron theorem, $|\beta_j|\le 1$, so we get 
$$
\prod_{i=1}^{r-2s-1} (1-\alpha_i)\prod_{j=1}^s (2-2{\rm Re}\beta_j)\ge \frac{N}{d^{r-1}}.
$$
Thus by the arithmetic and geometric mean inequality, 
$$
\sum_{i=1}^{r-2s-1} (1-\alpha_i)+\sum_{j=1}^s (2-2{\rm Re}\beta_j)\ge (r-s-1)\left(\frac{N}{d^{r-1}}\right)^{\frac{1}{r-s-1}}.
$$
But the left hand side is the trace of $1-b/d$, so it is $\le r$, as ${\rm Tr}(b)\ge 0$. Thus we get 
$$
r\ge (r-s-1)\left(\frac{N}{d^{r-1}}\right)^{\frac{1}{r-s-1}}.
$$
This implies that 
$$
d\ge N^{\frac{1}{r-1}}\left(1-\frac{s+1}{r}\right)^{1-\frac{s}{r-1}}, 
$$
as desired. 

(iii) Analyzing the function of $s$ in (ii) and using that $2s\le r-1$, we find easily that the worst case scenario is the maximal possible value of $s$, i.e., $s=\lfloor \frac{r-1}{2}\rfloor$. This gives the result. 
The last statement follows from the fact that 
$$
\left( \frac{\lceil \frac{r-1}{2}\rceil}{r}\right)^{\frac{\lceil \frac{r-1}{2}\rceil}{r-1}}\ge \frac{1}{2}.
$$
\end{proof} 

\section{Auxiliary lemmas}

\begin{lemma}\label{loca} Let $H$ be a local finite dimensional quasi-Hopf algebra, $M$ a finite dimensional $H$-module, and $a: M\to M^{**}$ a homomorphism. Then generalized eigenspaces of $a$ (regarded as a linear map $M\to M$) are $H$-submodules of $M$. In particular, if $M$ is indecomposable, then all eigenvalues of $a$ are the same. 
\end{lemma} 

\begin{proof} The associated graded ${\rm gr}(H)$ of $H$ with respect to the radical filtration 
is a Hopf quotient of an enveloping algebra, so it has $S^2=1$. Thus the operator $S^2-1$ strictly preserves 
the radical filtration. We have $ah=S^2(h)a$, so 
$$
(a-\lambda)hv=h(a-\lambda)v+(S^2-1)(h)av
$$
We claim that 
$$
(a-\lambda)^Nhv=\sum_{i=0}^N
\binom{N}{i}(S^2-1)^i(h)(a-\lambda)^{N-i}a^iv.
$$
The proof is by induction in $N$. 
Namely, the case $N=0$ is clear, and for the induction step note that 
$$
(a-\lambda)\sum_{i=0}^N
\binom{N}{i}(S^2-1)^i(h)(a-\lambda)^{N-i}a^iv=
$$
$$
=\sum_{i=0}^N
\binom{N}{i}(S^2-1)^i(h)(a-\lambda)^{N+1-i}a^iv+\sum_{i=0}^N
\binom{N}{i}(S^2-1)^{i+1}(h)(a-\lambda)^{N-i}a^{i+1}v,
$$
so the induction step follows from the identity $\binom{N}{i}+\binom{N}{i-1}=\binom{N+1}{i}$. 

Therefore, since $(S^2-1)^i(h)$ is zero for sufficiently large $i$, the generalized eigenspaces 
of $a$ are $H$-submodules of $M$. This implies the statement. 
\end{proof} 

\begin{remark} Note that in Lemma \ref{loca}, the ordinary eigenspaces of $a$ (as opposed to generalized eigenspaces) 
don't have to be submodules of $M$, e.g. $M=H$ and $a=S^2$ when $S^2\ne 1$. An example of a local Hopf algebra 
with $S^2\ne 1$  is given in \cite{NWW}, Theorem 1.3, case B3. 
\end{remark}

\begin{lemma}\label{commu} Let $V$ be a simple object of a tensor category $\C$, and 
$a: V\to V^{**}$ be an isomorphism. Then the morphism $a\otimes (a^*)^{-1}$ commutes 
with the composition of $a\otimes 1$, evaluation and coevaluation 
$$
V\otimes V^*\to V^{**}\otimes V^*\to \bold 1\to V\otimes V^*
$$ 
\end{lemma} 

\begin{proof} It is easy to see that the morphism $a\otimes (a^*)^{-1}$ 
preserves the evaluation and the coevaluation moprhism, so it commutes with their composition. 
It also obviously commutes with $a\otimes 1$. 
\end{proof} 

Now let $H$ be a finite dimensional Hopf algebra over an algebraically closed field $\k$ with antipode $S$
such that $S^2\ne 1$. 

\begin{lemma}\label{nontri} Let $M$ be a finite dimensional $H$-module which tensor generates $\Rep(H)$ and $a: M\to M^{**}$ be an isomorphism. Then $a$ is not a scalar (when regarded as a linear map $M\to M$). 
\end{lemma}

\begin{proof} Assume the contrary. Then for any $h\in H$ we have $h|_{M}=S^2(h)|_{M}$. Hence the same holds for any tensor product of copies of $M$ and $M^*$. But since $M$ tensor generates $\Rep(H)$, some such product contains the regular representation of $H$, which implies that $S^2=1$ on $H$, a contradiction. 
\end{proof} 

Recall that a Hopf algebra $H$ is called simple if it has no nontrivial Hopf quotients. 

\begin{lemma}\label{no2} Let $H$ be a simple finite dimensional Hopf algebra 
over an algebraically closed field $\k$ of characteristic $\ne 2$ with $S^2\ne 1$.
Let $V$ be an irreducible 2-dimensional $H$-module such that $V\cong V^{**}$. 
Then $S^4\ne 1$. In particular, the distinguished grouplike element of $H$
is nontrivial. 
\end{lemma} 

\begin{proof} Since $H$ is simple and has a 2-dimensional simple module, 
$H$ cannot have nontrivial 1-dimensional modules: otherwise $H$ would have a nontrivial commutative 
Hopf quotient. For the same reason, 
${\rm Ext}^1(\k,\k)=0$, since otherwise $H$ would have a 
nontrivial local Hopf quotient. Also it is clear that $V$ tensor generates $\Rep(H)$. 

Assume the contrary, i.e. that $S^4=1$. Pick an isomorphism $a: V\to V^{**}$ such that $a^2=1$ (this can be done since $S^4=1$). Then $a$ is semisimple with eigenvalues $\pm 1$ (as ${\rm char}(\k)\ne 2$) 
and not a scalar by Lemma \ref{nontri}, so 
${\rm Tr}(a)=0$. Since $H$ has no nontrivial 1-dimensional modules and ${\rm Ext}_H^1(\k,\k)=0$, 
this implies that the Loewy series of $V\otimes V^*$ must be $\bold 1,W,\bold 1$, where 
$W$ is an irreducible 2-dimensional module. By Lemma \ref{commu}, the 
morphism $a\otimes (a^*)^{-1}$ commutes with the composition of $a\otimes 1$, evaluation and coevaluation
$V\otimes V^*\to V\otimes V^*$. Hence the trace of the block of the morphism $a\otimes (a^*)^{-1}$ on the constituent $W$ 
is $-2$, which implies that $a\otimes (a^*)^{-1}$ must act by $=-1$ on $W$. But since $H$ is simple, 
$W$ must tensor generate $H$, which gives a contradiction
with Lemma \ref{nontri}. 
\end{proof}

\begin{example} Let $n\ge 3$, $q$ be a primitive $n$-th root of unity, and 
$H=\mathfrak{u}_q(\mathfrak{sl}_2)$ be the corresponding small quantum group 
over $\Bbb C$, generated by a grouplike element $K$ and skew-primitive 
elements $E$, $F$ with $K^n=1$, $E^n=0$, $F^n=0$ and the usual commutation relations
$KE=q^2EK$, $KF=q^{-2}EK$, $[E,F]=\frac{K-K^{-1}}{q-q^{-1}}$. 
This Hopf algebra has a 2-dimensional tautological representation
$V$, which is simple, self-dual and tensor generates $\Rep(H)$, 
Also $S^2(a)=KaK^{-1}$, so the order of $S^2$ is 
the order of $q^2$, which is $n$ for odd $n$ and $n/2$ for even $n$. 
Thus $S^4=1$ for $n=4$. However, $H$ is simple if and only 
if $n$ is odd: otherwise it has a 1-dimensional representation $\psi$ 
such that $\psi(E),\psi(F)=0$, $\psi(K)=-1$. So the only reason the proof 
of Lemma \ref{no2} fails for $n=4$ is that the 
representation $W$ is not irreducible but rather isomorphic to $\psi\oplus \psi$
and therefore does not tensor generate $\Rep(H)$. 
\end{example} 

\begin{lemma}\label{tensprodproj} 
Let $X$ be a simple object of a finite tensor category $\C$, such that $X\otimes X^*$ is projective. 
Then $X$ is projective. 
\end{lemma} 

\begin{proof} Assume the contrary. Then 
we have 
$$
P_X\otimes X^*=X\otimes X^*\oplus Y.
$$ 
Hence $\Hom(X,P_X)=\Hom(\bold 1,P_X\otimes X^*)\ne 0$. 
Thus $P_X$ is the injective hull of $X$, i.e. the socle of $P_X$ is $X$. 
Hence $P_X\otimes X^*=2X\otimes X^*\oplus Z$, so 
$\dim\Hom(X,P_X)=\dim\Hom(\bold 1,P_X,\otimes X^*)\ge 2$, a contradiction. 
\end{proof} 

\section{Finite dimensional quasi-Hopf algebras with two simple modules}

In this section we consider integral finite tensor categories of rank $r=2$ unless otherwise specified.

\subsection{The lower bound for the Frobenius-Perron dimension of a $\Bbb Z_+$-generator}
For $r=2$ the $\Bbb Z_+$-ring $A$ is $\Bbb Z_+$-generated by the nontrivial basis element $X$, so we may take $d={\rm FPdim}(X)$. Theorem \ref{bound} then gives the inequality $d\ge N/2$. More explicitly, let $b_0=1,b_1=X$. We have $X^2=aX+b$ for some $a\ge 0$, $b\ge 1$. Hence $d^2=ad+b$. It is easy to see that $p_0=d-a$ and $p_1=1$, so ${\rm FPdim}(A)=p_0+p_1d=2d-a$. Thus we get 

\begin{lemma}\label{l1} Let $A$ be a transitive $\Bbb Z_+$-ring of rank $2$, and let the dimension of the nontrivial basis element $X$ be $d$. Then 
$$
N-1\ge d\ge N/2,
$$ 
where $N={\rm FPdim}(A)$. 
\end{lemma} 

\subsection{The case of large characteristic}
Now let $\C=\Rep H$ be the representation category of a quasi-Hopf algebra $H$ over $\k$ with two simple modules. 
If $\C$ is semisimple, it is easy to see that $\C={\rm Vec}(\Bbb Z/2,\omega)$ is the category of $\Bbb Z/2$-graded vector spaces twisted by a $3$-cocycle $\omega$. So from now till the end of the section we assume that $\C$ is non-semisimple. 

Let $C$ be the Cartan matrix of $\C$. 

\begin{lemma}\label{detbound} $|\det C|\le \frac{{\rm FPdim}(\C)^2}{4d^2}$. 
\end{lemma}

\begin{proof} Using the arithmetic and geometric mean inequality, we have 
$$
|\det C|=|c_{00}c_{11}-c_{10}c_{01}|\le {\rm max}(c_{00}c_{11},c_{10}c_{01})\le 
$$
$$
\le \frac{1}{4d^2}{\rm max}((c_{00}+c_{11}d^2)^2,(c_{10}+c_{01})^2d^2)\le 
$$
$$
\le \frac{1}{4d^2}(c_{00}+(c_{01}+c_{10})d+c_{11}d^2)^2=\frac{{\rm FPdim}(\C)^2}{4d^2}. 
$$
\end{proof} 

Let $A$ be the Grothendieck ring of $\C$. 

\begin{proposition}\label{largechar} Suppose that the characteristic of $\k$ is $0$ 
or \linebreak $p>\frac{{\rm FPdim}(\C)^2}{4d^2}$, $\C$ is not pointed, and 
${\rm Id}\cong **$ as an additive functor.  
Then ${\rm FPdim}(\C)=m{\rm FPdim}(A)^2$, where $m>1$ is an integer. In particular, 
${\rm FPdim}(\C)$ cannot be square free. 
\end{proposition}

\begin{proof} By \cite{EGNO}, Theorem 6.6.1 and Lemma \ref{detbound}, $P_0$ and $P_1$ must be proportional in the Grothendieck ring
of $\C$, so $P_0=(d-a)P_1$. This implies that $P_1=m((d-a)\bold 1+X)$, where $m$ is a positive integer. 
Hence ${\rm FPdim}(\C)=\dim P_0+d\dim P_1=m(2d-a)^2$. 

Assume that $m=1$. Since $\C$ is not pointed, it is unimodular. Also we have 
$[P_1:X]=1$. By the Frobenius property of $\C$ this implies that $P_1=X$, so 
$d-a=0$, a contradiction. Thus, $m>1$, as desired. 
\end{proof} 

\begin{remark} 1. Note that the statement that $m>1$ in Proposition \ref{largechar} is false in the pointed case, e.g. 
for the representation category of Sweedler's 4-dimensional Hopf algebra.

2. Pointed finite tensor categories with two simple objects in characteristic zero 
are classified in \cite{EG2}. They are either $\Rep H$ where $H$ is a Nichols Hopf algebra of dimension $2^n$ (\cite{N}) or exceptional quasi-Hopf algebras of dimension $8$ or $32$. 
\end{remark} 

\subsection{Arbitrary characteristic} 

\begin{corollary}\label{proje} In any characteristic one has ${\rm FPdim}(\C)\ge d\cdot {\rm FPdim}(A)$, i.e., $D\ge d$, with the equality if and only if $X$ is projective. In particular, ${\rm FPdim}(\C) \ge \frac{1}{2}{\rm FPdim}(A)^2$. Moreover, if $X$ is not projective then 
$D\ge 2d$ and ${\rm FPdim}(\C) \ge {\rm FPdim}(A)^2$. 
\end{corollary}  

\begin{proof} This follows from Proposition \ref{p1}, since $p_1=1$. 
\end{proof} 

\subsection{Minimal categories} 
Let us say that $\C$ is {\it minimal} if $X$ is projective. 

\begin{proposition}\label{mini} Suppose $\C$ is minimal and $d>1$. Then $d=p^r$ where $p={\rm char}(\k)$, and $a=p^r-p^s$ for some $s\le r$, so ${\rm FPdim}(A)=p^{s}(p^{r-s}+1)$, ${\rm FPdim}(\C)=p^{r+s}(p^{r-s}+1)$, $D=p^r$ and $\dim P_0=b=p^{r+s}$. 
\end{proposition}

\begin{proof} We have $X^2=aX\oplus P_0$, and $P_0$ is an iterated extension of $\bold 1$, i.e. it coincides with the projective cover of $\bold 1$ in the unipotent subcategory of $\C$. This implies the statement, since the Frobenius-Perron dimension of any unipotent category is the power of the characteristic (Example \ref{unip}). 
\end{proof} 

\begin{corollary}\label{c1} Let $\C$ be a minimal category of Frobenius-Perron dimension $pn>2$ where 
$p$ is a prime and $p>n$. Then $p=n+1$ and $n$ is a power of $2$, i.e., 
$p=2^{2^m}+1$ is a Fermat prime. Moreover, ${\rm char}(\k)=2$, $d=p-1$ and $X^2=(p-2)X+p-1$. 
\end{corollary} 

\begin{proof} By Proposition \ref{mini} we have $pn=q^{r+s}(q^{r-s}+1)$, where 
$q$ is a prime. Since $p>n$, we must have $s=0$, $n=q^r$ and $p=q^r+1$. Thus $q$ is a power of $2$ and $p$ is 
a Fermat prime, as desired. 
\end{proof} 

\subsection{Examples}

\begin{example}
Minimal categories as in Corollary \ref{c1} exist for every Fermat prime $p$. Namely, let $G={\rm Aff}(p)$ be the group of affine linear transformations of the field $\Bbb F_{p}$, then we can take $\C={\rm Rep}_\k(G)$. 
This shows that Proposition \ref{largechar} fails in small characteristic. 
\end{example}

\begin{example}
Consider the case $a=0$, i.e., $X^2=d^2$ for some positive integer $d$. We claim that for $d>1$ such a ring $A$ 
admits a categorification by a finite tensor category in characteristic $p$ if and only if $d$ is a power of $p$, and in particular it never admits such categorification in characteristic $0$. 

Indeed, let $\C$ be a finite tensor category with Grothendieck ring $A$. We may assume that $\C$ is tensor generated by $X$. Then $\C=\C_0\oplus \C_1$ is $\Bbb Z/2$-graded and ${\rm Ext}^1(\bold 1,X)={\rm Ext}^1(X,\bold 1)=0$. Thus $P_{\bold 1}$ is an iterated extension of $\bold 1$ and $\dim P_{\bold 1}={\rm FPdim}(\C_0)$ is a power of $p$, since $\C_0$ is unipotent (Example \ref{unip}). Also, $P_X$ is an iterated extension of $X$, and $X\otimes P_X=P_{\bold 1}$. 
This implies that $d^2[P_X:X]=p^m$ for some $m$, which implies that $d=p^s$ for some $s$. 

On the other hand, such categories exist for every $s$.  Indeed, it is enough to construct such a category $\C$ 
for $s=1$, then for any $s$ we can take the subcategory $\C_s$ in $\C^{\boxtimes s}$ generated by $X^{\boxtimes s}$. 
For $p=2$, an example of $\C$ is given in Example \ref{X2=4}. If $p>2$, let $H$ be the Hopf algebra $\Bbb Z/2\ltimes \k[x,y]/(x^p,y^p)$, where $x,y$ are primitive and the generator $g\in \Bbb Z/2$ acts by $gxg^{-1}=-x, gyg^{-1}=-y$. 
Let $J=\sum_{j=0}^{p-1}\frac{x^j\otimes y^j}{j!}$ be a twist  for $H$, and $H^J$ be the corresponding twisted triangular Hopf algebra. Then we can take $\C$ to be the category of comodules over $H^J$. 
\end{example}

\subsection{Categorifications of $X^2=X+d(d-1)$}\label{a=1} Suppose $a=1$, i.e., $X^2=X+d(d-1)$. Let $\C$ be a finite tensor category with Grothendieck ring $A$. Since $X$ is self-dual, this implies that any morphism $a: X\to X^{**}$ has zero trace. Indeed, otherwise $X\otimes X=\bold 1\oplus Y$ where $Y$ has only one copy of $X$ in the composition series. Thus $\Hom(Y,\bold 1)\ne 0$ 
(as $Y$ is self-dual), so $\dim\Hom(X\otimes X,\bold 1)\ge 2$, a contradiction. 

\begin{proposition}\label{chardiv} The characteristic of $\k$ is a prime $p$ dividing \linebreak $d(d-1)$. Moreover, if $H$ 
is a Hopf algebra then $p$ divides $d$.  
\end{proposition} 

\begin{proof} Consider the Loewy series of $X\otimes X$. Exactly one of its terms contains $X$. 
So we have a 3-step filtration of $X\otimes X$ with composition factors $M,X,N$ invariant under any automorphism of $X\otimes X$, where $M,N$ are indecomposable iterated self-extensions of $\bold 1$ (in fact, $M$ and $N^*$ are cyclic). 
The $H$-modules $M$ and $N$ factor through a local quasi-Hopf algebra $H'$. Therefore, by Lemma \ref{loca}, for any isomorphism \linebreak $a: X\to X^{**}$ the morphism $a\otimes a$ on $X\otimes X$ has only one eigenvalue on $M$ and only one on $N$. Moreover, these two coincide by Lemma \ref{commu}; we will normalize $a$ in such a way that this eigenvalue is $1$.
Then, computing the trace of $a\otimes a$, we get $d(d-1)=0$ in $\k$, as claimed. 

Now assume that $H$ is a Hopf algebra and let $a_1,...,a_d$ be the eigenvalues of $a$. We may assume that $d>2$. Then the eigenvalues of $a\otimes a$ are $a_ia_j$. At least $d(d-1)$ of these numbers are $1$, so at most $d$ of them are $\ne 1$. We claim that $a_i=1$ for all $i$ or $a_i=-1$ for all $i$. Indeed, if none of $a_i$ are $1$ or $-1$ then $a_i^2\ne 1$ for all $i$, so $a_ia_j=1$ for $i\ne j$, hence $a_i=a_j^{-1}$, which is impossible since $d>2$. So at least one $a_i$ is $1$ or $-1$, and we can assume it is $1$ by multiplying $a$ by $-1$ if needed. Let $n>0$ be the number of $i$ such that $a_i=1$. If $a_i=1,a_j\ne 1$ then $a_ia_j, a_ja_i\ne 1$, so $2n(d-n)\le d$. Thus for $d>2$ we get $n=d$, i.e. $a_i=1$ for all $i$. Thus, $d={\rm Tr}(a)=0$ in $\k$, as desired.  
\end{proof} 

\subsection{Categorifications of $X^2=X+2$}\label{X2=X+2}
Now consider the case $d=2$, i.e. the fusion rule $X^2=X+2$. 
 Let $\C_X$ be the tensor subcategory of $\C$ generated by $X$. 
 
 \begin{proposition}\label{tenpro} 
 One has $X\otimes X=P\oplus X$, where $P$ is a nontrivial extension of $\bold 1$ by $\bold 1$. 
 Moreover, $P$ and $X$ are the indecomposable projectives of $\C_X$, so ${\rm FPdim}(\C_X)=6$. 
 \end{proposition} 
 
 \begin{proof} Let us first show that $X\otimes X=P\oplus X$. Assume the contrary. 
 Since $X\otimes X$ is self-dual and $\dim\Hom(\bold 1,X\otimes X)=1$, we get that 
 $X\otimes X$ must be indecomposable with Loewy series $\bold 1,X,\bold 1$. 
 Then $\dim\Hom(X,X\otimes X\otimes X)=\dim{\rm End}(X\otimes X)=2$. 
 But $X\otimes X\otimes X$ is self-dual, has composition series $X,X\otimes X,X$, and $\Hom(X,X\otimes X)=0$. 
 This means that $X\otimes X\otimes X=2X\oplus X\otimes X$. Thus for each $n$, $X^{\otimes n}$ 
 is a direct sum of copies of $X$ and $X\otimes X$. But for some $n$, the object $X^{\otimes n}$ 
 has a direct summand which is projective in $\C_X$. Hence $X\otimes X$ is projective. Therefore, so is $X$ 
 (as it is a direct summand in $X\otimes X\otimes X$). But then $X$ must be a direct summand in $X\otimes X$, a contradiction. 
 
 Note that $\dim\Hom(X,P\otimes X)=\dim\Hom(X\otimes X,P)=\dim\Hom(P,P)=2$. Thus, $P\otimes X=2X$. 
 Thus, for any $n$, $X^{\otimes n}$ is a direct sum of copies of $P$ and $X$. This implies that $P$ and $X$ are the indecomposable projectives of $\C_X$.  
 \end{proof} 
 
We note that such a category of dimension $6$ does exist, e.g. ${\rm Rep}_\k(S_3)$, with ${\rm char}(\k)=2$. It can also be realized as the $\Bbb Z/2$-equivariantization ${\rm Vec}_\k(\Bbb Z/3)^{\Bbb Z/2}$ of ${\rm Vec}_\k(\Bbb Z/3)$. A generalization 
is ${\rm Vec}_\k(\Bbb Z/3,\omega)^{\Bbb Z/2}$, where $\omega\in H^3(\Bbb Z/3,\k^\times)$. 

\subsection{Categories of dimension $sn$, where $s$ is square free and coprime to $n$}
\begin{proposition}\label{32} Let ${\rm char}(\k)$ be zero or $p>\frac{{\rm FPdim}(\C)^2}{4d^2}$, and 
${\rm FPdim}(\C)=sn$, where $s$ is a square free number coprime to $n$. Assume that $\C$ is not pointed, and ${\rm Id}\cong **$ as additive functors. Then $n=m\ell^2$, where $\ell\ge 4$ and $m\ge 2$, so $n\ge 32$.  
\end{proposition} 

\begin{proof} It follows from Proposition \ref{largechar} that $n=m\ell^2$ where $\ell={\rm FPdim}(A)$ and $m\ge 2$. 
So it remains to show that $\ell\ge 4$. To do so, note that $\ell\ge 2$, and if $\ell=2$ then $\C$ is pointed. 
Also, if $\ell=3$ then $X^2=X+2$, so ${\rm char}(\k)=2$ (Proposition \ref{chardiv}).  
\end{proof} 

\begin{example}\label{32dim} (see \cite{Ne}, Subsection 9.4) The bound $n\ge 32$ in Proposition \ref{32} is sharp. Namely, let ${\rm char}(\k)\ne 2$ and 
let $B$ be the small quantum group $\mathfrak{u}_q(\mathfrak{sl}_2)$ at a primitive 8-th root of unity $q$. 
It is a Hopf algebra of dimension $2^7=128$ generated by $E,F,K$ with relations $KE=q^2EK$, $KF=q^{-2}EK$, $[E,F]=\frac{K-K^{-1}}{q-q^{-1}}$, $E^4=F^4=K^8-1=0$. With its usual Hopf structure, $B$ does not admit an $R$-matrix, but 
it admits a quasi-Hopf structure $\Phi$ with a factorizable R-matrix (\cite{CGR}). Let $\D=\Rep(B^\Phi)$ 
be the corresponding nondegenerate braided category. This category contains a unique 1-dimensional nontrivial representation $\psi$ such that $\psi(E)=\psi(F)=0$, $\psi(K)=-1$ (it has highest weight $4$). It is easy to show using formula (4.10) of \cite{CGR} that the subcategory $\mathcal{E}$ generated by $\bold 1$ and $\psi$ has trivial braiding, i.e., it is equivalent to $\Rep(\Bbb Z/2)$. Let $\C:=\mathcal{E}^\perp/\mathcal{E}$ be the de-equivariantization of the centralizer $\mathcal{E}^\perp$ of $\mathcal{E}$ by $\Bbb Z/2$. The centralizer consists of representations with even highest weights ($0,2,4,6$), so modding out by $\psi$ gives a category $\C$ with just two simple objects, $\bold 1$ and $X$ (the quantum adjoint representation). 

The category $\C$ has Frobenius-Perron dimension $32$ and is nondegenerate braided. Its Grothendieck ring is generated by a basis element $X$ with $X^2=2X+3$. The indecomposable projectives $P_{\bold 1},P_X$ have the Loewy series $\bold 1,X\oplus X,\bold 1$ and 
$X,\bold 1\oplus \bold 1,X$, and we have $X\otimes X=\bold 1\oplus P_X$. 

Also $\C$ does not admit a pivotal structure: indeed, if $g$ is a pivotal structure and $u: {\rm Id}\to **$ the Drinfeld isomorphism then $g^{-1}u$ is an automorphism of the identity functor, hence acts by $1$ on $X$ (as $X$ is linked to $1$). 
But then the squared braiding $c^2: X\otimes X\to X\otimes X$ is unipotent, which is a contradiction with formula (4.10) of \cite{CGR}, as this formula implies that $-1$ is an eigenvalue of $c^2$ on the tensor product of highest weight vectors. 

Moreover, any {\it pivotal} finite nondegenerate braided tensor category $\C$ over $\Bbb C$ with two simple objects is equivalent to ${\rm Vec}(\Bbb Z/2)$ with braiding defined by the quadratic form $Q$ on $\Bbb Z/2$ such that $Q(1)=\pm i$. Indeed, by the result of \cite{GR}, a {\it pivotal} finite nondegenerate braided tensor category over $\Bbb C$ must contain a simple projective object, so $X$ is projective, which implies that $\C$ is semisimple. Thus, the above example shows that the pivotality assumption in \cite{GR} is essential. 
\end{example} 

\begin{remark} I don't know any other finite dimensional quasi-Hopf algebras over $\Bbb C$ with a unique nontrivial simple module $X$, such that $\dim X>1$. In particular, I don't know if there exist Hopf algebras with this property. 
\end{remark} 

\subsection{Non-minimal categories} 

\begin{corollary}\label{c3} (i) Suppose $\C$ is non-minimal and ${\rm FPdim}(\C)=pn$, where $p$ is a prime and $p>n$. Then ${\rm FPdim}(A)$ divides $n$ and the dimension of every projective object in $\C$ is divisible by $p$. 
\end{corollary} 

\begin{proof} By Corollary \ref{proje} we have ${\rm FPdim}(A)\le \sqrt{{\rm FPdim}(\C)}<p$, so since ${\rm FPdim}(A)$ divides $pn$, it must divide $n$. Hence $D$ is divisible by $p$, as desired. 
\end{proof} 

Corollary \ref{c3} has strong implications for the structure of $\C$ if $n$ is small. 
First of all note that if $\C$ is pointed then its Frobenius-Perron dimension has to be even. 

\begin{proposition}\label{p4} Let $\C$ be an integral category with two simple objects of Frobenius-Perron dimension $pn$ where $p$ is a prime. If $n\le 3$ then $n=2$ and $\C$ is pointed or ${\rm FPdim}(\C)=6$ and $\C$ is a minimal category which categorifies the fusion rule $X^2=X+2$ in characteristic $2$. 
\end{proposition}

\begin{proof} By Corollary \ref{c1}, we may assume that $\C$ is non-minimal. 

Suppose $n=1$. If $p>2$, Corollary \ref{c3} implies that ${\rm FPdim}(A)=1$, which is impossible, and if $p=2$, the category is semisimple, hence pointed.  

Suppose $n=2$. If $p>2$, Corollary \ref{c3} implies that ${\rm FPdim}(A)=2$, i.e., $\C$ is pointed. 
If $p=2$ then $d<2$, so $d=1$ and again $\C$ is pointed. 

Suppose $n=3$. If $p>3$, Corollary \ref{c3} implies that ${\rm FPdim}(A)=2d-a=3>d$, which implies that $d=2$, 
$a=1$, so $b=2$ and $X^2=X+2$. In the latter case, 
${\rm char}(\k)=2$ (Subsection \ref{X2=X+2}). Consider an isomorphism $g: P_{\bold 1}\to P_{\bold 1}^{**}$, and 
let $w\in H$ be the element realizing the canonical isomorphism $V\to V^{****}$ such that $\varepsilon(w)=1$. 
Then $g^{-2}w: P_{\bold 1}\to P_{\bold 1}$ is an automorphism, and we may normalize $g$ so that 
this automorphism is unipotent. Then the eigenvalues of $g$ on the invariant part 
of the graded pieces of the Loewy filtration of $P_{\bold 1}$ are all equal to $1$. 
Since the trace of any morphism $X\to X^{**}$ is zero, we see that ${\rm Tr}(g)$ 
cannot be zero since $\dim P_{\bold 1}=p$ and hence $[P_{\bold 1}:\bold 1]$ is odd. 
This is a contradiction, since this implies that $\bold 1$ is projective so $\C$ is semisimple. 
\end{proof} 

\subsection{Representations of finite groups}

Recall that ${\rm Aff}(q)$ denotes the group of affine transformation $x\to \alpha x+\beta$ over the finite field $\Bbb F_q$. 

\begin{theorem}\label{figr} Let $G$ be a finite group. Then:

(i) If $G$ has exactly two irreducible representations over $\k$ 
then either $G=\Bbb Z/2$ or ${\rm char}(\k)=p>0$. 

(ii) Let ${\rm char}(\k)=p>0$ and let $N$ be the largest normal $p$-subgroup of $G$. Then $G$ has exactly two irreducible representations over $\k$ if and only if $G/N$ is one of the following: 

(1) $G={\rm Aff}(q)$ where $q$ is a Fermat prime, and $p=2$;

(2) $G={\rm Aff}(p+1)$ where $p$ is a Mersenne prime. 

(3) $G={\rm Aff}(9)$ or $G=\Bbb Z/2\ltimes {\rm Aff}(9)$ (where the generator of $\Bbb Z/2$ acts by the Galois automorphism of $\Bbb F_9$) and $p=2$.

(4) $G=\Bbb Z/2$, $p\ne 2$. 
\end{theorem} 

\begin{proof} (i) Let ${\rm char}(\k)=0$ and $X$ be the nontrivial irreducible representation of $G$. Then $X\otimes X$ has to contain $\dim X$ copies of the trivial representation. Thus $\dim X=1$ and $G=\Bbb Z/2$. 

(ii) This follows by specializing Theorem A in \cite{DN} to the case when the only linear representation of $G$ is trivial
and using the easy fact that consecutive prime powers are $q-1,q$ for a Fermat prime $q$, 
or $p,p+1$ for a Mersenne prime $p$, or $8,9$. See e.g. \cite{CL} for a short proof of this statement. 
\end{proof} 

\section{Lower bounds for integral finite tensor categories with no nontrivial invertible objects} 

\begin{proposition} \label{lowbo} Let $\C=\Rep(H)$ 
be a non-semisimple integral finite tensor category of rank $r$ such that 
${\rm Ext}^1(\bold 1,\bold 1)=0$ (e.g., ${\rm char}(\k)=0$) and $\C$ has no nontrivial invertible objects. Then: 

(i) $\FPdim \C\ge 8r+3$;

(ii) If $r=2$ then $\FPdim(\C)\ge 32$. 

In particular, these bounds hold if $\C$ is not pointed and simple (has no nontrivial tensor subcategories). 
\end{proposition} 

\begin{proof} (i) It is clear that $\dim P_{\bold 1}\ge 4$, since it has head and socle $\bold 1$. 

Assume first that $\dim P_{\bold 1}=4$. Then $P_{\bold 1}$ is uniserial with 
Loewy series $\bold 1,V,\bold 1$ for some 2-dimensional simple module $V$, so $V^*\cong V$ and 
$\dim\ext^1(\bold 1,V)=1$, while ${\rm Ext}^1(\bold 1,Y)=0$ for any simple $Y\ne V$. 
So if $V\otimes V$ has Loewy series $\bold 1,Y,\bold 1$, then $Y$ must be simple, and 
we must have $Y\cong V$ (as otherwise ${\rm Ext}^1(\bold 1,Y)=0$). Then $V\otimes V=P_{\bold 1}$, so 
$V$ is projective by Lemma \ref{tensprodproj}, a contradiction.
\footnote{Another way to get a contradiction is to note that $V^2=V+2$ 
in the Grothendieck ring and apply Proposition \ref{tenpro}.}

Thus, $V\otimes V=\bold 1\oplus X$, where $X$ is a self-dual simple 3-dimensional module. 
Consider the tensor product $P_{\bold 1}\otimes V=P_V\oplus T$, which has composition factors $V,\bold 1, X,V$. 
Note that $P_V$ necessarily includes composition factors $V,\bold 1,V$. If $P_V$ contains no other composition factors, i.e. it is 5-dimensional, then $T=X$, so $X=P_X$ is projective. But then $X\otimes V=P_V\oplus Q$ (as $\Hom(X\otimes V,V)\ne 0$), so $Q$ is projective and $\dim Q=1$, which is impossible. Thus, $P_V$ is indecomposable 8-dimensional with Loewy series $V,\bold 1\oplus X,V$. 

Let $W$ be a 2-dimensional projective module. Then $W^*\otimes W=P_{\bold 1}$, so 
$\End(W\otimes V)=\Hom(W^*\otimes W,V\otimes V)=\Hom(P_{\bold 1},\bold 1\oplus X)=\k$. 
But $\C$ is unimodular, so the head and socle of any indecomposable projective coincide, and hence if $P$ is 
a non-simple indecomposable projective then $\dim\End(P)\ge 2$. 
Thus, $W\otimes V=Q$ is a 4-dimensional simple projective module. Moreover, 
$Q\otimes V=W\otimes V\otimes V=W\oplus W\otimes X$, and 
$$
{\rm End}(W\otimes X)=\Hom(W^*\otimes W,X\otimes X)=\Hom(P_{\bold 1},X\otimes X)=
\Hom(X\otimes P_{\bold 1},X).
$$
But ${\rm Ext}^1(V,X)\ne 0$, so $P_X$ includes composition factors
$X,V,X$, hence $\dim P_X\ge 8$. Thus $X\otimes P_{\bold 1}$, which has dimension $12$, 
 cannot contain more than one copy of $P_X$.
Hence ${\rm End}(W\otimes X)=\Hom(X\otimes P_{\bold 1},X)=\k.$ 
Thus $W\otimes X$ is a 6-dimensional simple projective module. So $W$ can be recovered from $Q$ as the unique 2-dimensional summand in $Q\otimes V$, i.e., 
the assignment $W\mapsto Q$ is injective. 

Now suppose we have $s$ projective 2-dimensional objects. 
Then they contribute $4s$ to ${\rm FPdim}(\C)$, and the corresponding modules $Q$ contribute $16s$. 
The rest of the simple modules except $\bold 1,V,X$ contribute at least $8$ each, so in total $8(r-2s-3)$. 
So altogether we have (using that $\dim P_{\bold 1}=4$, $\dim P_V=8$, $\dim P_X\ge 8$):
$$
{\rm FPdim}(\C)\ge 4+2\cdot 8+3\cdot 8+4s+16s+8(r-2s-3)=8r+20+4s\ge 8r+20,
$$
as claimed.  

Now assume that $\dim P_{\bold 1}\ge 5$. Then by Lemma \ref{easybound}, for each 
2-dimensional simple module $W$ we have $\dim P_W\ge 5/2$, so $W$ is not projective, and
$\dim P_W\ge 4$. 

If $\dim P_{\bold 1}\ge 5$ then let $X$ be a simple submodule of $P_{\bold 1}/\bold 1$.
If $\dim X\ge 3$ then $\dim P_X\ge 7$ (as it has composition factors $X,\bold 1,X$). Hence ${\rm FPdim}(\C)\ge 5+3\cdot 7+8(r-2)=8r+10$, as claimed. 

It remains to consider the case $\dim X=2$, in which case we must have $\dim P_{\bold 1}\ge 6$. 
Then $\dim P_X\ge 5$. If $\dim P_X=5$ then the Loewy series of $P_X$ is $X,\bold 1,X$, 
so $X^*\otimes P_X$ has composition series $X^*\otimes X,X^*,X^*\otimes X$. But 
$X^*\otimes P_X=P_{\bold 1}\oplus Q$, and $P_{\bold 1}$ has composition series $\bold 1,X,X',\bold 1$ where $\dim X'=2$, and $Q$ is a 
4-dimensional projective, so cannot involve $\bold 1$. Thus $X^*\otimes X=\bold 1\oplus Y$, where $Y$ is a 3-dimensional simple, 
i.e. we have only one 2-dimensional simple constituent in $X^*\otimes P_X$, a contradiction. Thus, $\dim P_X\ge 6$. 
If $\dim P_X>6$ or $\dim P_{\bold 1}>6$ then they contribute at least $19$ into ${\rm FPdim}(\C)$, so ${\rm FPdim}(\C)\ge 19+8(r-2)=8r+3$, as claimed. 

Thus it remains to consider the case $\dim P_{\bold 1}=\dim P_X=6$. Then the Loewy series of $P_X$ must have the form 
$X,\bold 1\oplus \bold 1,X$ (so ${\rm Ext}^1(\bold 1,X)=\k^2$) and the Loewy series of $P_{\bold 1}$ therefore looks like $\bold 1,X\oplus X,\bold 1$, where $X$ is self-dual. 
Now consider the module $X\otimes X$. If it has Loewy series $\bold 1,X',\bold 1$ where $X'$ is simple 2-dimensional then $\Hom(P_{\bold 1},X\otimes X)=\k$ (as then $X\otimes X$ is not a quotient of $P_{\bold 1}$), which is a contradiction since we have $[X\otimes X,\bold 1]=2$. 
Thus, $X\otimes X=\bold 1\oplus Y$, where $Y$ is a 3-dimensional simple. So $Y$ contributes at least $9$ into ${\rm FPdim}(\C)$. 
Thus we get ${\rm FPdim}(\C)\ge 6+2\cdot 6+9+8(r-3)=8r+3$, as claimed. This proves (i).

To prove (ii), note that $\dim P_{\bold 1}$ is divisible by ${\rm dim}P_X$. Let $X$ be the nontrivial simple object. If $\dim X\ge 4$ then $\dim P_X\ge 9$, so $\dim P_{\bold 1}\ge 9$ and $\FPdim(\C)\ge 9+4\cdot 9=45$, as claimed. If $\dim X=3$ then $\dim P_X\ge 7$, so $\dim P_{\bold 1}\ge 7$, which means that $[P_{\bold 1}:X]\ge 2$, 
so $\dim P_{\bold 1}\ge 8$, $\dim P_X\ge 8$ and ${\rm FPdim}(\C)\ge 32$, as claimed. 

Finally, the case $\dim X=2$ is impossible by Proposition \ref{tenpro}, as ${\rm Ext}^1(\bold 1,\bold 1)\ne 0$. 
\end{proof} 

\begin{example} 1. The bound ${\rm FPdim}(\C)\ge 27$ for $r=3$ in Proposition \ref{lowbo} is sharp in any characteristic. Namely, it is attained for the category 
$\C$ of representations of the small quantum group $H=\mathfrak{u}_q(\mathfrak{sl}_2)$ where $q$ is a primitive cubic root of $1$ if ${\rm char}(\k)\ne 3$ and of the restricted enveloping algebra $H=\mathfrak{u}(\mathfrak{sl}_2)$ in characteristic $3$. 
In this case there are three simple modules: $\bold 1$, the 2-dimensional tautological module $V$ and the 3-dimensional Steinberg 
module $X$, which is projective. Moreover, the projective covers of $\bold 1$ and $V$ have dimension $6$ 
and Loewy series $\bold 1,V\oplus V,\bold 1$ and $V,\bold 1\oplus \bold 1,V$. 

2. The bound ${\rm FPdim}(\C)\ge 32$ for $r=2$ in Proposition \ref{lowbo} is also sharp, as shown by Example \ref{32dim}. 

3. There exists a category as in Proposition \ref{lowbo} with $\dim P_{\bold 1}=4$. Namely, let $\k=\Bbb C$, 
$G$ be the binary icosahedral group (of order $120$), and $V$ its 2-dimensional tautological representation. Let $z$ be the nontrivial central element of $G$. Consider the supergroup $G\ltimes V$ and the category $\C={\rm Rep}(G\ltimes V,z)$ of representations of $G$ on superspaces such that $z$ acts by parity. Then $P_{\bold 1}=\wedge V={\rm Ind}_G^{G\ltimes V}\bold 1$. The simple objects of this category are just simple $G$-modules, which are labeled by vertices of the affine Dynkin diagram $\widetilde{E}_8$ and have dimensions $1,2,3,4,5,6,3,4,2$, and ${\rm FPdim}(\C)=480$. 

4.  Examples of categories of small dimension ${\rm FPdim}(\C)=60$ satisfying the assumptions of Proposition \ref{lowbo} are given by the representation categories $\Rep_p(A_5)$ of the alternating group $A_5$ in characteristics $p=2,3,5$. Namely: 

$\Rep_2(A_5)$ has simple objects of dimension $1,2,2,4$ with projective covers of dimension $12,8,8,4$;

$\Rep_3(A_5)$ has simple objects of dimension $1,3,3,4$ with projective covers of dimension $6,3,3,9$;
The Loewy series of $P_\bold 1$ is $\bold 1,X,\bold 1$, where 
$X$ is the 4-dimensional simple module. 

$\Rep_5(A_5)$ has simple objects of dimension $1,3,5$ with projective covers of dimension $5,10,5$.
The Loewy series of $P_\bold 1$ is $\bold 1,X,\bold 1$, where 
$X$ is the 3-dimensional simple module. 

In particular, these examples show that $P_{\bold 1}$ can have Loewy series $\bold 1,V,\bold 1$ with $V$ of dimensions 
$2,3,4$. 
\end{example} 

\section{Tensor categories of dimension $p$} 
Now let $\C$ be a finite tensor category over an algebraically closed field $\k$ of any characteristic $q\ge 0$ of Frobenius-Perron dimension $p>2$ (a prime). It is known that such a category is automatically integral (\cite{EGNO}). If $\C$ is pointed then either it is semisimple (hence $\C=\Vec(\Bbb Z/p,\omega)$) or $\C$ is unipotent and ${\rm char}(\k)=p$. 

It is conjectured that $\C$ is always pointed. The goal of this section is to prove it under some assumptions. 

Suppose $\C$ is not pointed, and let $A$ be the Grothendieck ring of $\C$. Let $r$ be its rank, and $d$ the smallest dimension of a nontrivial simple object (call it $X$). Clearly $d>1$. Let $d_i$ be the dimensions of simple objects and $p_i$ the dimensions of the corresponding indecomposable projectives. So $\sum_i p_id_i=p$. Also, since ${\rm FPdim}(A)>1$ and divides $p$, we must have ${\rm FPdim}(A)=p$ and $D=1$. Also we have ${\rm Ext}^1(\bold 1,\bold 1)=0$, 
since otherwise the tensor subcategory of $\C$ formed by iterated extensions of $\bold 1$ is nontrivial, and its dimension 
must divide $p$ by \cite{EGNO}, Theorem 7.17.6.

\begin{proposition} \label{fourob}
(i) $\C$ has at least $4$ simple objects. 

(ii) $p\ge 37$. 
\end{proposition} 

\begin{proof} 
(i) Suppose $\C$ has two simple objects. Then by Theorem \ref{bound}, $d\ge p/2$, so $p\ge 1+d^2\ge 1+p^2/4$, hence 
$p\le 2$, a contradiction. 

Now assume that $\C$ has three simple objects, $\bold 1$ and $X,Y$ of Frobenius-Perron dimensions $d,m$. 
By Theorem \ref{bound}, we have $d\ge (p/3)^{1/2}$, so $3d^2\ge p$. 
If $P_X\ne X$ then $\dim P_X\ge 2d$ so 
$$
p\ge 1+2d^2+m^2>3d^2\ge p,
$$ 
a contradiction. Similarly, if $P_Y\ne Y$ then $\dim P_Y\ge 2m$ so 
$$
p\ge 1+d^2+2m^2>3d^2\ge p,
$$ 
again a contradiction. Thus $X$ and $Y$ are projective. 
This means that the block of $\bold 1$ is a tensor subcategory of $\C$. 
The Frobenius-Perron dimension of this subcategory must divide $p$, so this subcategory must be trivial. 
Hence $\C$ is semisimple. Then $X^2=1+aX+bY$ with $b\ne 0$, so $d^2=1+ad+bm$, so $(d,m)=1$. Now
$X\otimes Y=sX+tY$, so $dm=sd+tm$, hence $sd=(d-t)m$. But $s,t>0$ since $X$ occurs in $Y\otimes Y$ and $Y$ occurs in $X\otimes X$. This gives a contradiction. 

Thus $\C$ has at least $4$ simple objects. 

(ii) It follows from (i) and Proposition \ref{lowbo} that $p\ge 8\cdot 4+3=35$. Hence $p\ge 37$. 
\end{proof} 

\begin{corollary}\label{bound2} One has 
$$
d\ge 2^{-2/3}p^{\frac{1}{r-1}}.
$$
\end{corollary} 

\begin{proof} 
This follows from Proposition \ref{fourob}, Theorem \ref{bound}(iii) and the fact that the function 
$$
f(r):=\left( \frac{\lceil \frac{r-1}{2}\rceil}{r}\right)^{\frac{\lceil \frac{r-1}{2}\rceil}{r-1}}
$$
satisfies $f(r)\ge f(4)=2^{-2/3}$ if $r\ge 4$. 
\end{proof} 

\begin{proposition}\label{5sim} Let $H$ be a quasi-Hopf algebra of prime dimension $p$ with $r$ simple modules, where $r>1$ is not of the form $4k+1$ with $k\in \Bbb Z_+, k\ge 2$ (e.g., $r\le 8$). Then there exists a simple $H$-module 
$X\ne \bold 1$ such that $X\cong X^{**}$. 
\end{proposition} 

\begin{proof} Assume the contrary. Then the distinguished invertible object $\chi$ of $H$ must be $\bold 1$ (as $\chi^{**}\cong \chi$), so $X^{****}\cong X$ for all $X$ by categorical Radford's formula (\cite{EGNO}, Subsection 7.19). Thus the group $\Bbb Z/4$ acts freely on nontrivial simple $H$-modules by taking the dual. This implies that $r=4k+1$. 

It remains to show that $r\ne 5$. Indeed, for $r=5$ the nontrivial simple objects are 
$X_1=X$, $X_2=X^*$, $X_3=X^{**}$, $X_4=X^{***}$, so $d_j=d$ and $p_j=p_1$ for $1\le j\le 4$.
Consider the product $P_0\otimes X^*$, of dimension $dp_0$. It decomposes in a direct sum of $P_j$, $1\le j\le 4$. 
Thus we have $dp_0=np_1$, where $n$ is the number of summands. But $p_0$ and $p_1$ are coprime (since $p_0+4dp_1=p$), so $d$ is divisible by $p_1$. Hence $d=p_1$ and $X_j$ are projective for $1\le j\le 4$. Thus ${\rm Ext}^1(\bold 1,\bold 1)$ 
cannot be zero, a contradiction. 
\end{proof} 

\section{Hopf algebras of dimension $p$} 

Now let $H$ be a Hopf algebra of dimension $p$ (a prime) over a field $\k$ of characteristic $q$. 
It was recently proved by R. Ng and X. Wang \cite{NW} that if $4q\ge p$ or $q=2$ then $H={\rm Fun}(\Bbb Z/p)$ 
or $q=p$ and $H=\k[x]/x^p$ with $\Delta(x)=x\otimes 1+1\otimes x$ or $\Delta(x)=x\otimes 1+1\otimes x+x\otimes x$. 
This is conjectured to hold for any $q$, and our goal is to improve the bound $4q\ge p$. 

Assume that $q>2$ and $H$ is not of this form. If $H$ is semisimple and cosemisimple then by \cite{EG3}, Theorem 3.4, 
$H$ is commutative and cocommutative, so this is impossible. Thus, by replacing $H$ with $H^*$ if needed (as done in \cite{NW}), we may (and will) assume that $H$ is not semisimple. Also, by a standard argument, $H$ and $H^*$ have no nontrivial $1$-dimensional modules (as their order must divide $p$). In particular, the distinguished group-like elements of $H$ and $H^*$ must be trivial , so $S^4=1$. Thus $S^2\ne 1$ (as $\dim H=0$ in $\k$) and defines an involution on the simple $H$-modules $X_i$, sending $X_j$ to $X_j^{**}$. Also, $H$ is simple, so ${\rm Rep}(H)$ is tensor generated by any nontrivial simple module and ${\rm Ext}^1(\bold 1,\bold 1)=0$.

Let $r$ be the number of simple $H$-modules, $d_j$ be the dimension of $X_j$ and $p_j$ be the dimension of the projective cover $P_j$ of $X_j$. 

\begin{proposition}\label{rankbound} (i) (cf. \cite{EG3}, Lemma 2.10) There exists $i$ such that $X_i^{**}=X_i$ and $p_i$ 
and $d_i$ are both odd. Moreover, if there exists $j\ne 0$ with $X_j\cong X_j^{**}$ then 
there exists $i\ne 0$ with $p_i$ odd. 

(ii) (\cite{NW}, Lemma 2.3) Suppose $p_i$ is odd and $X_i^{**}=X_i$. Then $p_i\ge q+2$. 

(iii) Let $d_*$ be the maximum of all $d_j$, and let $i$ be such that $X_i\cong X_i^{**}$ and $p_i$ is odd, with largest possible $d_i$. Then we have 
$$
p\ge (q+2)\left(d_i+\frac{1}{d_i}\sum_{k\ne i}{\rm max}(d_id_k/d_*,1)\right).
$$ 

(iv) If $d_i=d_*$ then 
$$
p\ge (q+2)\left(d_i+\frac{\sum_{k\ne i}d_k}{d_i}\right).  
$$ 
In particular,
$$
p\ge (q+2)\left(d_i+\frac{2r-3}{d_i}\right).  
$$ 
If $d_i<d_*$ then 
$$
p\ge (q+2)\left(d_i+1+\frac{r-2}{d_i}\right).
$$ 

(v) One has $p\ge 2(q+2)\sqrt{r-1}$. 

(vi) If $d_i>1$ then $d_i\ge 3$. 
\end{proposition} 

\begin{proof} (i) Let $J$ be the set of $j$ such that $X_j^{**}\cong X_j$. Since double dual is an involution, we have 
that $\sum_{j\notin J}p_jd_j$ is even. Hence $\sum_{j\in J}p_jd_j$ is odd (as they add up to $p$). So there exists 
$i$ such that $p_id_i$ is odd, which proves the first statement. 

To prove the second statement, we may assume that $d_j$ or $p_j$ is even, otherwise 
we can take $i=j$. We may also assume that $p_0$ is odd; otherwise $i\ne 0$ automatically. 
Now consider the even-dimensional object $X_j^*\otimes P_j=P_0\oplus \oplus_{k\ne 0}m_kP_k$. We see that 
$\sum_{k\ne 0}m_kp_k$ is odd.  So there exists $k$ such that $p_k$ is odd, and moreover, there is an odd number of 
such $k$. Hence the involution $**$ has a fixed point on the set of such $k$, which gives the statement. 

(ii) Let $a: P_i\to P_i^{**}$ be an isomorphism. We may choose $a$ so that $a^2=1$. Then ${\rm Tr}(a)=0$, otherwise $\bold 1$ is a direct summand in $P_i$, hence is projective and $H$ is semisimple, a contradiction. 
Let $m_+,m_-$ be the multiplicities of the eigenvalues $1$ and $-1$ for $a$. Then $m_++m_-=p_i$, 
$m_+-m_-=sq$ for some integer $s$. It is clear that $s$ has to be odd, so $|m_+-m_-|\ge q$, and $m_-,m_+>0$ by Lemma \ref{nontri}. Thus $p_i\ge q+2$.

(iii) By Lemma \ref{easybound}, $p_k\ge N_{ji}^kp_i/d_j\ge N_{ji}^k(q+2)/d_j$ for any $j,k$. 
It is clear that for each $k$ there exists $j$ such that $X_k$ is a composition factor of $X_j\otimes X_i$ (so that $N_{ji}^k\ge 1$) and $d_j\le d_id_k$. Thus 
$$
p_kd_k\ge (q+2){\rm max}(d_id_k/d_*,1)/d_i.
$$ 
Hence 
$$
p=p_id_i+\sum_{k\ne i} p_kd_k\ge (q+2)\left(d_i+\frac{1}{d_i}\sum_{k\ne i}{\rm max}(d_id_k/d_*,1)\right),
$$
as desired. 

(iv) The first inequality follows immediately from (iii). The second inequality follows from the first one, as $d_k\ge 2$ for 
$k\ne 0$. To prove the third inequality, take $j$ such that $d_j=d_*$, and separate the corresponding term in the sum in (iii), while replacing all the other terms by $1$. This gives the statement. 

(v) It follows from (iv) that 
$$
p\ge (q+2)\left(d_i+\frac{r-1}{d_i}\right), 
$$ 
so the statement follows by the arithmetic-geometric mean inequality. 

(vi) This follows from Lemma \ref{no2}. 
\end{proof} 

The following corollary improves the bound $p>4q$ from \cite{NW}. 

\begin{corollary}\label{9/2} One has $p>\frac{14}{3}(q+2)$.  
\end{corollary}

\begin{proof} By Proposition \ref{fourob}, $r\ge 4$. Assume first that $r=4$. It is clear that $X^{**}=X$ for any 
simple object $X$ (otherwise we would have objects $X,X^*,X^{**},X^{***}$ 
which are pairwise nonisomorphic, which would give $r\ge 5$). Thus by 
Proposition \ref{rankbound}(i), we have $d_i\ge 2$, hence by Proposition \ref{rankbound}(vi)
$d_i\ge 3$. Thus, by Proposition \ref{rankbound}(iv) we get $p>14(q+2)/3$.

It remains to consider the case $r\ge 5$. In this case, if $d_i\ge 2$ then by Proposition \ref{rankbound}(vi) $d_i\ge 3$, so 
if $d_i=d_*$ then by Proposition \ref{rankbound}(iv) we have $p>16(q+2)/3$, and 
if $d_i\ne d_*$ then Proposition \ref{rankbound}(iv) yields $p>5(q+2)$.
On the other hand, if $d_i=1$ then Proposition \ref{rankbound}(iv) yields 
$p>r(q+2)\ge 5(q+2)$. 

So in all cases we have $p>14(q+2)/3$, which proves the statement. 
\end{proof} 

\begin{proposition}\label{inner} Suppose that for some $i$ we have $X_i\cong X_i^{**}$, $p_i$ is odd and $d_i>1$. 
Then 
$$
\frac{p}{q+2}> 2^{-2/3}p^{\frac{1}{r-1}}.
$$
\end{proposition} 

\begin{proof} By Corollary \ref{bound2}, we have 
$$
d_i\ge 2^{-2/3}p^{\frac{1}{r-1}}. 
$$
Thus by Proposition \ref{rankbound}(ii) we have 
$$
p> p_id_i\ge (q+2)d_i\ge 2^{-2/3}(q+2)p^{\frac{1}{r-1}},
$$
which implies the statement. 
\end{proof} 

\begin{corollary}\label{inner1} (i) Let $H$ be a non-semisimple Hopf algebra of dimension $p$ in characteristic $q$ 
for which there exists $i\ne 0$ such that $X_i^{**}\cong X_i$ (for instance, $**$ is isomorphic to ${\rm Id}$ as an additive functor on $\Rep(H)$, i.e., $S^2$ is an inner automorphism). 
Then if
$$
\frac{p}{q+2}\le  2^{-2/3}p^{\frac{1}{r-1}}.
$$
then $H$ is commutative and cocommutative. 

(ii) Let $W$ be the Lambert $W$-function (the inverse to the function $f(z)=ze^z$), and 
$$
\phi(x):=2^{-2/3}\exp\left(\frac{1}{2}W(2^{13/3}\log x)\right). 
$$ 
If 
$$
\frac{p}{q+2}\le \phi(p)
$$
then $H$ is commutative and cocommutative. 
\end{corollary} 

\begin{proof} 
(i) By Proposition \ref{inner}, it suffices to show that there exists $i$ such that $X_i^{**}\cong X_i$, 
$d_i>1$ and $p_i$ is odd. But this follows from Proposition \ref{rankbound}(i). 

(ii) Fix $p$ and note that the function $2\sqrt{x}$ is increasing and $2^{-2/3}p^{1/x}$ is decreasing 
in $x$. Thus if $s$ is the solution of the equation 
$$
2\sqrt{s}=2^{-2/3}p^{1/s}
$$
then by Proposition \ref{rankbound}(v) and part (i), 
$$
\frac{p}{q+2}>{\rm max}(2\sqrt{x},2^{-2/3}p^{1/x})|_{x=r-1}\ge 2\sqrt{s}.
$$
Let $y=2^{5/3}\sqrt{s}$. We have 
$$
y=p^{2^{10/3}/y^2}.
$$ 
Thus 
$$
2^{13/3}\log p=y^2\log(y^2).
$$
So $\log(y^2)=W(2^{13/3}\log p)$.
Thus we get 
$$
y=\exp\left(\frac{1}{2}W(2^{13/3}\log p)\right), 
$$
so
$$
2\sqrt{s}=2^{-2/3}\exp\left(\frac{1}{2}W(2^{13/3}\log p)\right).
$$
This implies (ii). 
\end{proof} 

\begin{remark} It is easy to show that 
$$
\phi(x)\sim 2\sqrt{2} \left(\frac{\log x}{\log \log x}\right)^{1/2}
$$
as $x\to \infty$. 
\end{remark} 

\begin{theorem}\label{inner2} 
Let $H$ be a Hopf algebra of prime dimension $p$ over a field $\k$ characteristic $q>2$, which is not commutative and cocommutative. Then
$$
\frac{p}{q+2}>{\rm max}\left(\frac{14}{3},{\rm min}(9,\phi(p))\right).
$$
\end{theorem}  

\begin{proof} If there exists simple $X\ne \bold 1$ with $X\cong X^{**}$ then Corollary \ref{inner1} applies. 
Otherwise, Proposition \ref{5sim} and Proposition \ref{rankbound}(v) imply that 
$p> r(q+2)\ge 9(q+2)$. This together with Corollary \ref{9/2} implies the theorem. 
\end{proof}

\end{document}